\documentclass[11pt]{article}

\usepackage{amsmath}
\usepackage{amssymb}
\usepackage{amsthm}
\usepackage{url}
\usepackage{textcomp}
\usepackage{graphicx}
\usepackage[mathlines]{lineno}
\usepackage{cite}
\usepackage{enumerate}
\usepackage{mathtools}
\usepackage{thm-restate}

\usepackage{color}

\newtheorem{theorem}{Theorem}[section]
\newtheorem{lemma}[theorem]{Lemma}
\newtheorem{proposition}[theorem]{Proposition}
\newtheorem{corollary}[theorem]{Corollary}

\newtheorem{definition}[theorem]{Definition}
\newtheorem{observation}[theorem]{Observation}
\newtheorem{remark}[theorem]{Remark}
\newtheorem{example}[theorem]{Example}
\newtheorem*{NPprob}{Problem}
\newtheorem{question}[theorem]{Question}
\newtheorem{problem}[theorem]{Problem}
\newtheorem{conjecture}[theorem]{Conjecture}

\newenvironment{thm}{\begin{theorem}}{\end{theorem}}
\newenvironment{lem}{\begin{lemma}}{\end{lemma}}
\newenvironment{prop}{\begin{proposition}}{\end{proposition}}
\newenvironment{cor}{\begin{corollary}}{\end{corollary}}

\newenvironment{rem}{\begin{remark}\bgroup\rm }{\egroup\end{remark}}
\newenvironment{ex}{\begin{example}\bgroup\rm }{\egroup\end{example}}

\newenvironment{qstn}{\begin{question}\bgroup\rm }{\egroup\end{question}}

\newcommand{\spec}{\operatorname{spec}}
\newcommand{\match}{\alpha'}

\newcommand{\tr}{\operatorname{tr}}
\newcommand{\diam}{\operatorname{diam}}
\newcommand{\dunion}{\dot{\cup}}

\newcommand{\lam}{\lambda}

\newcommand{\bit}{\begin{itemize}}
\newcommand{\eit}{\end{itemize}}
\newcommand{\ben}{\begin{enumerate}}
\newcommand{\een}{\end{enumerate}}
\newcommand{\beq}{\begin{equation}}
\newcommand{\eeq}{\end{equation}}
\newcommand{\bea}{\begin{eqnarray*}}
\newcommand{\eea}{\end{eqnarray*}}
\newcommand{\bpf}{\begin{proof}}
\newcommand{\epf}{\end{proof}\ms}
\newcommand{\ms}{\medskip}

\newcommand{\du}{\,\dot{\cup}\,}

\title{Note on von Neumann and R\'enyi entropies of a Graph}
\author{
Michael Dairyko\thanks{Department of Mathematics, Iowa State University, Ames, IA 50011, USA. (mdairyko@iastate.edu, hogben@iastate.edu, chlin@iastate.edu,  myoung@iastate.edu)} \and
Leslie Hogben\footnotemark[1]\phantom{$^*$}\thanks{American Institute of
Mathematics, 600 E. Brokaw Rd., San Jose, CA 95112, USA. (hogben@aimath.org)} \and 
Jephian C.-H.~Lin\footnotemark[1] \and 
Joshua Lockhart  \thanks{Department of Computer Science, University College London, Gower Street, London WC1E 6BT, United Kingdom. (joshua.lockhart.14@ucl.ac.uk, davideroberson@gmail.com, simoseve@gmail.com)}\and
David Roberson\footnotemark[3] \and
Simone Severini\footnotemark[3] \and 
Michael Young\footnotemark[1]
}

\begin{document}
\maketitle

\vspace{-10pt}
\begin{abstract} 
We conjecture that all connected graphs of order $n$ have von Neumann entropy at least as great as the star $K_{1,n-1}$ and prove this for  almost all graphs of order $n$.  We show that connected graphs of order $n$ have R\'enyi 2-entropy at least as great as  $K_{1,n-1}$ and for $\alpha>1$, $K_n$  maximizes R\'enyi $\alpha$-entropy     over graphs of order $n$. We show that adding an edge to a graph can lower its von Neumann entropy. \end{abstract}




\section{Introduction}\label{sintro} 

 In the density matrix formulation of quantum mechanics, the state of a physical system is represented by a positive semi-definite matrix with unit trace called its
\emph{density matrix}. The \emph{von Neumann entropy} of a quantum state is defined in terms of the eigenvalues of its density matrix, and provides a means of characterizing its information content, in analogy to the Shannon entropy of a statistical ensemble from classical information theory. Indeed, the von Neumann entropy of a state takes center stage in the burgeoning field of quantum information theory \cite{NC}.

It is well known that the combinatorial Laplacian matrix $L$ of a finite simple graph is positive semi-definite, and so the matrix $\frac 1 {\tr L}L$ (which has unit trace) can be interpreted as the density matrix of a physical system. It is therefore natural to interpret the von Neumann entropy of such a density matrix as the von Neumann entropy of the corresponding graph, with a view towards characterizing the information content of the graph \cite{BGS06,PS08}.

In this note we study graphs that minimize or maximize von Neumann entropy and its well known generalization, the R\'enyi $\alpha$-entropy, over (connected) graphs of fixed order.  We show (Theorem \ref{densitythm}) that almost all graphs of order $n$ have von Neumann entropy at least as great as the star $K_{1,n-1}$, all connected graphs of order $n$ have R\'enyi 2-entropy at least as great as  $K_{1,n-1}$ (Theorem \ref{2Renstarmin}), and for $\alpha>1$ all graphs of order $n$ have R\'enyi $\alpha$-entropy no greater than that of the complete graph $K_n$ (Corollary \ref{cor:Rmax}); it is known that $K_n$ maximizes the von Neumann entropy.   We also answer negatively a question from \cite{PS08} about the effect of adding an edge on von Neumann entropy (Proposition \ref{addedgelowerent}).  The von Neumann entropy and R\'enyi $\alpha$-entropies of a graph are defined precisely below. 

\ms 

The  {\em Shannon entropy} of a discrete probability distribution $p = (p_1, \ldots, p_n)$ is defined to be \[S(p):=\sum_{i=1}^n p_i\log_2\frac{1}{p_i}=-\sum_{i=1}^n p_i\log_2 {p_i},\]
with $0\log_20$ defined to be zero.
Let $G$ be a graph that has at least one edge.  Consider the (combinatorial) Laplacian scaled to have trace one, $\rho(G):=\frac 1{\tr L(G)}L(G)$, where  $L(G)=D(G)-A(G)$ with $D(G)$ the diagonal matrix of degrees and $A(G)$ the adjacency matrix.  The \textit{von Neumann entropy} of $G$ is defined to be the Shannon entropy of the probability distribution represented by the eigenvalues of $\rho(G)$,
\[S(G):=\sum_{i=1}^n\lambda_i\log_2\frac{1}{\lambda_i}=-\sum_{i=1}^n\lambda_i\log_2 {\lambda_i},\]
where $\{\lambda_i\}_{i=1}^n$ is the spectrum of $\rho(G)$ (multiset of eigenvalues), which is denoted by $\spec(\rho(G))$.  

For $\alpha \ge 0$ and $\alpha \ne 1$, the R\'enyi entropy of a discrete probability distribution $p = (p_1, \ldots, p_n)$ is defined as\[H_\alpha(p) = \frac{1}{1-\alpha} \log_2\left(\sum_{i = 1}^n p_i^\alpha \right);\]
this is also called the R\'enyi $\alpha$-entropy. The limit as $\alpha \to 1$ of $H_\alpha(p)$ is the Shannon entropy $S(p)$, so  as done in \cite{Renyi} we define $H_1(p)=S(p)$. Since $H_\alpha(p)$ is a non-increasing function of $\alpha$  for a fixed $p$  \cite{RenyiXE},  $S(p)\ge H_\alpha(p)$ for $\alpha\ge  1$.  

For a positive semidefinite matrix $M$ with trace 1, we  define $S(M)$ (respectively,  $H_\alpha(M)$) to be equal to the Shannon entropy (respectively, the R\'enyi $\alpha$-entropy) of the probability distribution given by the eigenvalues of $M$.  For a graph $G$, we  define $H_\alpha(G)=H_\alpha(\rho(G))$,  the R\'enyi $\alpha$-entropy of the scaled Laplacian.  The R\'enyi 2-entropy  is a useful tool in the study of von Neumann entropy, and R\'enyi $\alpha$-entropy is interesting in its own right.
\ms

The graphs realizing the minimum and maximum von Neumann entropy over all graphs on $n$ vertices are known, but  minimizing over connected graphs is still an open question.  A graph $G$ has zero von Neumann entropy if and only if 
one eigenvalue is 1 and the rest are 0.  These spectra are achieved only by graphs of the form 
$K_2\dunion \overline{K_{n-2}}$.

\begin{prop}\label{Smaxmin}
{\rm \cite{BGS06}} 
For all graphs on $n$ vertices, the maximum von Neumann entropy is attained by $K_n$ with $S(K_n)=\log_2(n-1)$,  and the minimum von Neumann entropy of $0$ is attained by $K_2\dunion \overline{K_{n-2}}$. 
\end{prop}

In \cite{PS08} it was asked whether the star minimizes von Neumann entropy among connected graphs of fixed order, and we conjecture this.
\begin{conjecture}\label{Neumannmin}
For all connected graphs on $n$ vertices, the minimum von Neumann entropy is attained by $K_{1,n-1}$.
\end{conjecture}

This conjecture is confirmed by \textit{Sage} up to $8$ vertices; when $G$ is restricted to being  a tree, it is true up to $15$ vertices \cite{sagedata1}.  In Theorem \ref{densitythm} we show that it is true for almost all graphs as $n\to\infty$ by use of the R\'enyi 2-entropy.  We also make a conjecture about trees.
\begin{conjecture}
For all trees on $n$ vertices, the maximum von Neumann entropy is attained by $P_n$.
\end{conjecture}

This conjecture is confirmed by \textit{Sage} up to $15$ vertices \cite{sagedata1}.


It is well known (and easy to show) that \vspace{-3pt}
\[\spec(L(K_{a,b}))=\left\{a+b,b^{(a-1)},a^{(b-1)},0\right\},\vspace{-3pt}\]
where $\lambda^{(m)}$ denotes the fact that $\lambda$ has multiplicity $m$.
The next result then follows by computation.
\begin{prop}
\label{starprop}
For complete bipartite graphs, the von Neumann entropy is \vspace{-3pt}
\[S(K_{a,b}) = 1 + \frac{b+1}{2b}\log_2 a + \frac{a+1}{2a}\log_2 b - \frac{a+b}{2ab}\log_2(a+b).\vspace{-3pt}\]
Specifically, for stars\vspace{-3pt}
\[S(K_{1,n-1})=\log_2(2n-2)-\frac{n}{2n-2}\log_2 n.\vspace{-3pt}\]
\end{prop}


Building graphs from pieces is a standard technique and it is useful to have information about the effect  of graph operations and constructions on von Neumann entropy.  Let $G$ be a graph.  Define $d_G$ to be the sum of degrees of all vertices, which is equal to the trace of the  combinatorial Laplacian and also equal to twice the number of edges in $G$.  
In the case of a disjoint union, we can determine the von Neumann entropy of the whole exactly from the entropies of the pieces.

\begin{prop}
Let $G_1, \ldots, G_k$ vertex-disjoint graphs and let\vspace{-3pt}
\[c_i=\frac{d_{G_i}}{\sum_j d_{G_j}}.\vspace{-3pt}\]
Then \vspace{-3pt}
\[S\left(\dot{\bigcup}_{i=1}^k G_i\right) = \sum_{i=1}^k c_i S(G_i) + \sum_{i=1}^k c_i \log_2 \frac{1}{c_i}.\vspace{-3pt}\]
\end{prop}
\begin{proof}
Let $\spec(\rho(G_i))=\{\lambda_j(i) : j \in [n_i]\}$ for $i \in [k]$. Then 
\[\spec\left(\dot{\bigcup}_{i=1}^k G_i\right) = \bigcup_{i=1}^k \left\{c_i \lambda_j (i): j \in [n_i]\right\}.\]
Therefore, the von Neumann entropy is 
\bea
S\left(\dot{\bigcup}_{i=1}^k G_i\right) &=& \sum_{i=1}^k \sum_{j=1}^{n_i} c_i \lambda_j(i) \log_2\frac{1}{c_i\lambda_j(i)}\\
 &=& \sum_{i=1}^k \left(c_i \left(\sum_{j=1}^{n_i} \lambda_j(i) \log_2\frac{1}{\lambda_j(i)}\right) + c_i \log_2\frac{1}{c_i}\left(\sum_{j=1}^{n_i} \lambda_j(i) \right)\right)\\
 &=& \sum_{i=1}^k c_i S(G_i) + \sum_{i=1}^k c_i \log_2\frac{1}{c_i}.\qedhere
\eea
\end{proof}

One way to think of the expression for $S(\dot{\cup}G_i)$ given in the previous proposition is the following: the first summation is a convex combination of the von Neumann entropies of the $G_i$ with coefficients $c_i$, and the second summation is the Shannon entropy of the probability distribution $(c_1, \ldots, c_k)$.

\begin{thm}\label{addedge}
{\rm \cite{PS08}} If $G$ and $H$ are two graphs on the same vertex set and $E(G)\cap E(H)=\emptyset$, then 
\[S(G\cup H)\geq \frac{d_G}{d_G+d_H}S(G)+\frac{d_H}{d_G+d_H}S(H).\]
In particular, if  $G$ is a graph and $e\in E(\overline{G})$, then 
\[S(G+e)\geq \frac{d_G}{d_G+2}S(G).\]
\end{thm}

The question of whether the factor $\frac{d_G}{d_G+2}$ can be removed from the second statement was raised in \cite{PS08}.

\begin{qstn}\label{addedgeQ}{\rm \cite{PS08}} Is the von Neumann entropy 
monotonically non-decreasing 
under edge addition?
\end{qstn}

We show in  Proposition \ref{addedgelowerent}  
 that adding an edge can decrease the von Neumann entropy slightly, answering Question \ref{addedgeQ} negatively.


\section{Using R\'enyi 2-entropy as a lower bound for von Neumann entropy}

In this section we give a lower bound for the von Neumann entropy in terms of the degree sequences of graphs by using the R\'enyi 2-entropy, using the fact that for all graphs $G$,
\beq\label{SgeH2}
S(G)\ge H_2(G).
\eeq

\begin{rem}\label{Rentropyform}
The R\'enyi 2-entropy of a trace one positive semidefinite  matrix $M$ can be expressed in the following useful manner:
\[H_2(M) = -\log_2\left(\sum_{i=1}^n\lam_i^2\right)= -\log_2 \tr(M^2) = -\log_2 \text{sum}(M \circ M),\]
where $\circ$ denotes the entrywise product (also called the Hadamard or Schur product) and $\text{sum}(M)$ is the sum of the entries of $M$. For a graph $G$ with vertex degrees $d_i$ for $i \in [n]$ and degree sum $d_G$, the R\'enyi 2-entropy of its scaled Laplacian $\rho(G)$ is 
\beq\label{R2eq} H_2(G)=-\log_2\left(\frac{d_G + \sum_i d_i^2}{d_G^2}\right)=\log_2\left(\frac{d_G^2}{d_G + \sum_i d_i^2}\right).\eeq
\end{rem}


\begin{thm}[R\'enyi-Quantum Star Test]
\label{quadracor}
Let $G$ be a graph on $n$ vertices satisfying 
\beq\label{eq:renshanstar}\frac{d_G^2}{\sum_{i=1}^n d_i^2+d_G}\geq \frac{2n-2}{n^{\frac{n}{2n-2}}}.\eeq
Then $S(G)\geq H_2(G)\geq S(K_{1,n})$.
\end{thm}
\begin{proof}
Recall that by Proposition \ref{starprop},\vspace{-3pt}
\[S(K_{1,n})=\log_2(2n-2)-\frac{n}{2n-2}\log_2 n=\log_2\left(\frac{2n-2}{n^{\frac{n}{2n-2}}}\right).\vspace{-3pt}\]
The result then follows immediately from \eqref{SgeH2}, \eqref{R2eq}, and the fact that $\log_2x$ is increasing. \end{proof}


As shown in Table \ref{quadratbl}, most graphs of small orders pass the R\'enyi-Quantum Star Test;  the graphs that fail the R\'enyi-Quantum Star Test  are shown in  \cite{sagedata1} for $n\le 8$.
 
\begin{table}[h!]
\[\begin{array}{c|ccc}
n & \text{\# } H_2(G)<S(K_{1,n}) & \text{\# connected graphs} & \text{percentage} \\
\hline 
2 & 0 & 1 & 0.00 \\
3 & 1 & 2 & 0.50 \\
4 & 2 & 6 & 0.33 \\
5 & 4 & 21 & 0.19 \\
6 & 8 & 112 & 0.071 \\
7 & 16 & 853 & 0.019 \\
8 & 49 & 11117 & 0.0044 \\
9 & 106 & 261080 & 0.00041 \\
10 & 307 & 11716571 & 0.000026 \\
\end{array}\]
\caption{Number of graphs with $H_2(G)<S(K_{1,n})$}
\label{quadratbl}
\end{table}

All the graphs that fail the R\'enyi-Quantum Star Test are quite sparse, which led us to the following result.


\begin{thm}\label{densitythm}
Let $G$ be a graph on $n$ vertices and $m$ edges with 
\beq\label{densityeq} \frac{1}{\sqrt{n} -1} \le \frac m {{n \choose 2}},\eeq
i.e., having density at least $\frac{1}{\sqrt{n} -1}$. Then $S(G)\geq H_2(G)\geq S(K_{1,n})$.
As $n\to\infty$, almost all graphs satisfy \eqref{densityeq}.
\end{thm}
\begin{proof}
Theorem 1 from \cite{Caen} gives the bound that $\sum_{i = 1}^n d_i^2 \le m\left(\frac{2m}{n-1} + n -2\right).$  \vspace{-3pt}

\bea
{n \choose 2}\left( \frac{1}{\sqrt{n} -1}\right) &\le& m \\
2n(n-1) &\le& 4m(n^{1/2} -1) \\
2(n-1)m \left( \frac{2m}{n-1} + n \right) &\le& 4m^2n^{1/2} \\
2(n-1)\left(m \left( \frac{2m}{n-1} + n-2\right)+2m \right) &\le& 4m^2n^{1/2} \\
2(n-1) \left(\sum_{i=1}^n d_i^2 + d_G\right) &\le& d_G^2 n^{1/2}. 
\eea
 
Since $\frac{n}{2n -2}\ge \frac 1 2$,   \[\frac{2(n-1)}{n^{\frac{n}{2n -2}}} \le \frac{2(n-1)}{n^{1/2}}\le \frac{d_G^2}{\sum_{i=1}^n d_i^2 + d_G}.\] Thus we have satisfied the condition for Theorem \ref{quadracor} and thus  $H_2(G)\ge S(K_{1,n-1})$.

By \cite[Theorem 3.2]{HHMS10}, if one chooses a graph $G$ at random from all labeled graphs on $n$ vertices, then $|E(G)|\geq \frac 1 2\binom{n}{2}-n\sqrt{2\ln n}$ with probability at least $1-n^{-2}$, justifying the last statement.  \end{proof}

Sufficient density implies a graph satisfies the R\'enyi-Quantum Star Test, but the converse is false. As an example, consider the path $G=P_n$, for which  $d_G=2n-2$ and $\sum_{i=1}^nd_i^2=4(n-2)+2=4n-6$.  Thus the left hand side of  \eqref{eq:renshanstar} is 
\[\frac{(2n-2)^2}{4n-6+(2n-2)}\sim \Theta(n);\]
while the right hand side is $\Theta(n^\frac{1}{2})$.  So for large enough $n$, the R\'enyi-Quantum Star Test shows $S(P_n)\geq H_2(P_n)\ge S(K_{1,n})$. In fact, for $n\ge 6$, $H_2(P_n)\ge S(K_{1,n-1})$.  This observation and Theorem \ref{quadracor} provide evidence for Conjecture \ref{Neumannmin}.


One could naturally ask whether the inequality $S(G)\ge H_\alpha(G)$ is tight ($\alpha>1$), and for what graphs.  
For a given probability distribution $p = (p_1, \ldots, p_n)$, the R\'enyi $\alpha$-entropy can be written as 
\bea
H_\alpha(p) 
 & = &-\frac{1}{\alpha-1}\log_2(p_1\cdot p_1^{\alpha-1}+p_2\cdot p_2^{\alpha-1}+\cdots +p_n\cdot p_n^{\alpha-1})\\
 & \leq& -\frac{1}{\alpha-1}\left[p_1\log_2(p_1^{\alpha-1})+p_2\log_2(p_2^{\alpha-1})+\cdots +p_n\log_2(p_n^{\alpha-1})  \right] \\
 & =& S(p).
\eea
It follows from the strict convexity of $-\log_2$ that $H_\alpha(p)$ is strictly less than $\sum_{i} -p_i\log_2(p_i)$ if and only if the nonzero $p_i$ are not all the same. Of course the latter quantity is just the Shannon entropy.   Hence $S(K_n)=H_\alpha(K_n)$, and this is the only connected graph on $n$ vertices that has $S(G)=H_\alpha(G)$ for $\alpha>1$.



\section{R\'enyi entropy}

For a fixed $\alpha$, it is natural to ask which graph(s)   maximize $H_\alpha(G)$ among graphs on $n$ vertices, and which graph(s)   minimize $H_\alpha(G)$, among graphs on $n$ vertices and among connected graphs on $n$ vertices.

\begin{prop}\label{Hminmax}
Fix  $\alpha>1$ and an integer  $n\ge 1$.  Over all probability distributions $p=(p_1,\dots,p_n)$: 
\ben
\item The  distribution $p_0=(1,0,\dots,0)$ minimizes $H_\alpha(p)$ and this is the only probability distribution (up to permutation of the entries) that does so.
\item The constant distribution $p_c=(\frac 1 n,\dots, \frac 1 n)$ maximizes $H_\alpha(p)$.
\een 
\end{prop}
\begin{proof}
It is clear that $0\leq H_\alpha(p)$ for all probability distributions $p$, and the only probability distribution that achieves $\alpha$-entropy zero is $p_0$. 

Now consider $p=(p_1,\dots,p_n)$. For all $\alpha>1$, $x^\alpha$ is a convex function, so by Jensen's inequality, 
\[\left( \frac 1{n}\right)^\alpha=\left(\sum_{i=1}^{n} \frac {p_i}{n}\right)^\alpha\le \sum_{i=1}^{n} \frac 1 {n} {p_i}^\alpha=\frac 1 {n}\sum_{i=1}^{n}  {p_i}^\alpha.\]
Thus $\sum_{i=1}^np^\alpha$ attains its minimum when $p_1=\cdots =p_{n}=\frac{1}{n}$, and  so $-\log_2\left(\sum_{i=1}^np_i^\alpha\right)$ attains its maximum there. 
\end{proof}

\begin{cor}\label{cor:Rmax}
Let $\alpha>1$.  For all (possibly disconnected) graphs $G$ on $n$ vertices, 
\[0=H_\alpha(K_2\du \overline{K_{n-2}})
\leq H_\alpha(G)\leq H_\alpha(K_n)=\log_2(n-1).\]
Furthermore, $K_2\du \overline{K_{n-2}}$ is the only graph that minimizes  R\'enyi $\alpha$-entropy for $\alpha >1$. \end{cor}
\begin{proof}
  It is known that 
$\spec(\rho(K_2\dot{\cup}\overline{K_{n-2}}))=\left\{1,0^{(n-1)}\right\}$ and this is the only graph on $n$ vertices that realizes this spectrum.  
  Therefore, 
  $K_2\dot{\cup} \overline{K_{n-2}}$ is the only graph on $n$ vertices with minimum R\'enyi $\alpha$-entropy.

Observe that $\spec(\rho(K_n))=\left\{\frac{1}{n-1}^{(n-1)},0\right\}$, so $H_\alpha(K_n)=\log_2(n-1)$. Since any graph $G$ has at least one Laplacian eigenvalue equal to zero,    $H_\alpha(G)\le H_\alpha(K_n)$ by Proposition \ref{Hminmax}.  
\end{proof}

For the minimum over connected graphs, we make the following conjecture.

\begin{conjecture}\label{Renyimin}
Let $\alpha>1$.  For any connected graph $G$ on $n$ vertices, 
\[H_\alpha(K_{1,n-1})\leq H_\alpha(G).\]
\end{conjecture}

The conjecture has been checked for $\alpha=1.1,1.5, 5,10$ for up to $8$ vertices by \textit{Sage} using code in \cite{sagedata1}, and is proved for $\alpha=2$ in Theorem \ref{2Renstarmin} below.  Notice that since $\lim_{\alpha\rightarrow 1^+}H_\alpha(G)=S(G)$, Conjecture \ref{Renyimin} implies Conjecture \ref{Neumannmin}.

\ms
The relationship between $H_2(G)$ and $H_2(H)$ can be described in terms of the degrees of the vertices of $G$ and $H$.  Let $d_1,\ldots ,d_n$ be the degree sequence of $G$. 
 Define 
\[\tr_2(G):=\frac{\sum_{i=1}^nd_i^2+d_G}{d_G^2}=\tr(\rho(G)^2).\]
From \eqref{R2eq}, $H_2(G)=-\log_2\left(\tr_2(G)\right)$, so $\tr_2(G)\geq \tr_2(H)$ if and only if $H_2(G)\leq H_2(H)$.  Therefore, Conjecture \ref{Renyimin} for $\alpha=2$ is equivalent to saying $\tr_2(G)\leq\tr_2(K_{1,n-1})$ for all connected graphs $G$ on $n$ vertices. 

The base case of the proof involves trees, and is proved by using the notion of majorization.  Let $\gamma=\{c_i\}_{i=1}^n$ and $\beta=\{b_i\}_{i=1}^n$ be two sequences of nonnegative integers with $\sum_{i=1}^nc_i=\sum_{i=1}^nb_i$.  Assuming the numbers are labeled such that 
$c_1\geq c_2\geq \cdots \geq c_n\text{ and }b_1\geq b_2\geq \cdots \geq b_n,$
we say that $\gamma$ \textit{majorizes} $\beta$ if  for all $k$ \vspace{-4pt}
\[\sum_{i=1}^kc_i\geq \sum_{i=1}^kb_i, \vspace{-4pt}\]
where the majorization is said to be strict if one of the inequalities is strict.  The next proposition is well known (and  easy to prove  from the definition).

\begin{prop}
Let $\gamma=\{c_i\}_{i=1}^n$ and $\beta=\{b_i\}_{i=1}^n$. 
 If $\gamma$ majorizes $\beta$, then  $\sum_{i=1}^nc_i^2\geq \sum_{i=1}^nb_i^2$, and the inequality is strict if the majorization is strict.
\end{prop}

\begin{prop}\label{treeR2star}  Among trees on $n$ vertices, the star $K_{1,n-1}$ is the unique tree that attains the minimum R\'enyi $2$-entropy, and the path $P_n$ is the unique tree that attains the maximum R\'enyi $2$-entropy.
\end{prop} 
\bpf For fixed $n$, let $\gamma=\{d_i\}_{i=1}^n$ be the degree sequence of a tree $T$, in non-increasing order.  Since the degree sum for every tree is  equal to $2n-2$, it is enough to show that the degree sequence of $K_{1,n-1}$ strictly majorizes the degree sequence of any other tree and the degree sequence of any other tree strictly majorizes the degree sequence of $P_n$.

Since $1\leq d_i$ for all $i$ and $\sum_{i=1}^nd_i=2n-2$, \vspace{-4pt}
\[\sum_{i=1}^kd_i\leq (2n-2)-(n-k) \vspace{-4pt}\]
and $K_{1,n}$ is the only tree that attains all equality, so $H_2(K_{1,n})<H_2(T)$ for all trees $T$ except $K_{1,n}$ itself.  

On the other hand, every tree has at least two leaves, so $d_{n-1}=d_n=1$.  Under this condition, $P_n$ is the only graph such that $\{d_i\}_{i=1}^{n-2}$ is evenly distributed.  Hence every other sequence strictly majorizes the degree sequence of $P_n$, so $H_2(P_n)>H_2(T)$ for all trees $T\ne P_n$.
\epf

The next result is well known (and straightforward to prove).
\begin{lem}\label{avelem}
Let $\{s_i\}_{i=1}^k$ and $\{t_i\}_{i=1}^k$ be positive real numbers.  Then
\[\min_i\left\{\frac{s_i}{t_i}\right\}\leq 
\frac{\sum_{i=1}^ks_i}{\sum_{i=1}^kt_i}\leq 
\max_i\left\{\frac{s_i}{t_i}\right\}.\]
If the ratios $\frac{s_i}{t_i}$ are not constant, then both inequalities are strict.
\end{lem}


\begin{lem}\label{addedgelem}
Let $G$ be a connected graph and $e\in E(\overline{G})$.  If $\tr_2(G)\leq \tr_2(K_{1,n-1})$, then $\tr_2(G+e)< \tr_2(K_{1,n-1})$.
\end{lem}
\begin{proof}
Assume that $e=uv$ with $\deg_G u=a$ and $\deg_G v=b$.  Let $\{d_i\}_{i=1}^n$ be the degree sequence of $G$.  Then 
\bea
\tr_2(G+e)&=&\frac{\left(2a+2b+2+\sum_{i=1}^nd_i^2\right)+\left(2+\sum_{i=1}^nd_i\right)}{\left(\sum_{i=1}^nd_i\right)^2+4\left(\sum_{i=1}^nd_i\right)+4}\\
&=&\frac{\left(\sum_{i=1}^nd_i^2+\sum_{i=1}^nd_i\right)+\left(2a+2b+4\right)}{\left(\sum_{i=1}^nd_i\right)^2+\left(4+4\sum_{i=1}^nd_i\right)}.
\eea
Next we show that 
\[\frac{2a+2b+4}{4+4\sum_{i=1}^nd_i}< \tr_2(K_{1,n-1})=\frac{1}{4}+\frac{3}{4(n-1)},\] by showing
\[2a+2b+4< \left(1+\frac{3}{n-1}\right)\left(1+\sum_{i=1}^nd_i\right).\]
Since $G$ must have at least $a+b$ edges, $\sum_{i=1}^nd_i\geq 2a+2b$; also, since $G$ is connected, $\sum_{i=1}^nd_i\geq 2(n-1)$.  Thus
\bea
\left(1+\frac{3}{n-1}\right)\left(1+\sum_{i=1}^nd_i\right)&=& 1+\left(\sum_{i=1}^nd_i\right)+\frac{3}{n-1}+\left(\frac{3}{n-1}\sum_{i=1}^nd_i\right)\\
&>& 1+2a+2b+\frac{3\cdot 2(n-1)}{n-1}\\
&>&2a+2b+4.\vspace{-3pt}
\eea
Now by Lemma \ref{avelem} and the assumption $\tr_2(G)\leq \tr_2(K_{1,n-1})$, we know $\tr_2(G+e)< \tr_2(K_{1,n-1})$.
\end{proof}

\begin{thm}\label{2Renstarmin}
Let $G$ be a connected graph on $n$ vertices other than $K_{1,n-1}$.  Then 
\[H_2(K_{1,n-1})< H_2(G).\]
\end{thm}
\begin{proof}
Since every connected graph has a spanning tree as a subgraph, by Theorem \ref{treeR2star} and Lemma \ref{addedgelem} we have $\tr_2(G)< \tr_2(K_{1,n-1})$.  Consequently, $H_2(K_{1,n-1})< H_2(G)$.
\end{proof}


\section{Comparison of von Neumann entropy and graph operations and parameters}\label{scompare}

In this section we examine the effect on von Neumann entropy of adding an edge and show von Neumann entropy is not comparable to many graph parameters. The next result shows that adding an edge is able to decrease von Neumann entropy slightly, providing a negative answer to Question \ref{addedgeQ}, which was first asked  in \cite{PS08}.  It is straightforward to verify.

\begin{prop}\label{addedgelowerent}
  Let  $v$ and $u$ be the two vertices of degree $n-2$ in $K_{2,n-2}$, and define the edge $e=vu$.  Then:
  \begin{enumerate}
  \item $\spec(\rho(K_{2,n-2})) = \left\{\frac{n}{4n-8}, \frac{n-2}{4n-8}, \frac{1}{2n-4}^{(n-3)}, 0\right\}$ and\\
  $\spec(\rho(K_{2,n-2}+e)) = \left\{\frac{n}{4n-6}^{(2)}, \frac{1}{2n-3}^{(n-3)}, 0\right\}$
\item $S(K_{2,n-2}) = \frac 1 2 + \frac n{4 n - 8}\log_2\frac{4 n - 8}n + 
   \frac{n - 3}{2 n - 4} \log_2( 2 n - 4)$ and\\ 
  $S(K_{2,n-2}+e) =  \frac n{2 n - 3}\log_2\frac{4 n - 6}n + 
   \frac{n - 3}{2 n - 3} \log_2( 2 n - 3)$
\end{enumerate}
For $n\ge 5$, $S(K_{2,n-2})>S(K_{2,n-2}+e)$.  \end{prop}

Proposition \ref{addedgelowerent} gives a family of graphs $K_{2,n-2}$ such that the ratio of $S(K_{2,n-2}+e)/S(K_{2,n-2})$ to $\frac{d_G}{d_G+2}$ goes to $1$ as $n$ goes to infinity;  thus in the asymptotic sense the inequality is tight.

On the other hand, an examination of the proof   \cite[Proposition 3.1]{OP93} for density matrices and its extension to graphs in \cite{PS08} shows that the inequalities in Theorem  \ref{addedge} are strict unless the density matrices of $G$ and $H$ are identical, which cannot happen for non-identical graphs (isomorphism does not suffice).  Therefore, for any graphs $G$ and $H$  with disjoint edge sets, the inequalities in Theorem \ref{addedge} are always strict.


Inspired by ``algebraic connectivity augmentation'' of a graph, the computational complexity of which is explored in \cite{MA08}, we define the following decision problem.
\begin{NPprob}
\textsc{EntropyAugmentation}\\
Input: \emph{A graph $G=(V,E)$, a non-negative integer $k$, a positive real number $x\in\mathbb{R}^+$}.\\
Output: \emph{YES if and only if there exists a subset $A\in E(\overline{G})$ of size $|A|\le k$ such that the von Neumann entropy of the augmented graph $S((V,E+A))\ge x$.}
\end{NPprob}
Since algebraic connectivity augmentation is NP-complete, we suggest that by similar reasoning it may be possible to prove that this problem is NP-hard. Its inclusion in NP is of course trivial, the certificate being the edge $e$ that ``augments'' the entropy by the required amount. We leave the following question open.
\begin{qstn} Is \textsc{EntropyAugmentation} an NP-complete decision problem?
\end{qstn}

{We have tried} to get von Neumann entropy to behave in concert with other graph parameters for a fixed number of vertices and edges.  For example, it was suggested  that $\match(G)<\match(H)$ implies $S(G)<S(H)$, where $\match(G)$ is the matching number, but this is not true (see Example \ref{mSno} below).  Von Neumann entropy and diameter  are noncomparable (see Example \ref{diamSno} below), and von Neumann entropy and maximum degree  are also noncomparable (see Example \ref{maxSno} below). 

\begin{ex}\label{mSno}  
Let $G_1$ and $G_2$ be the graphs shown in Figure \ref{fig:mSno}.  Then $\match(G_1)=2< 3=\match(G_2)$, but $S(G_1)\approx 1.94466 > 1.94188 \approx S(G_2)$.  Examples with the  reverse relation are easy to find, such as $\match(K_{1,3})=1< 2=\match(P_4)$ and $S(K_{1,3})\approx 1.25163<  1.31888\approx S(P_4)$. \vspace{-5pt}
\begin{figure}[!ht]
\begin{center}
\scalebox{.4}{\includegraphics{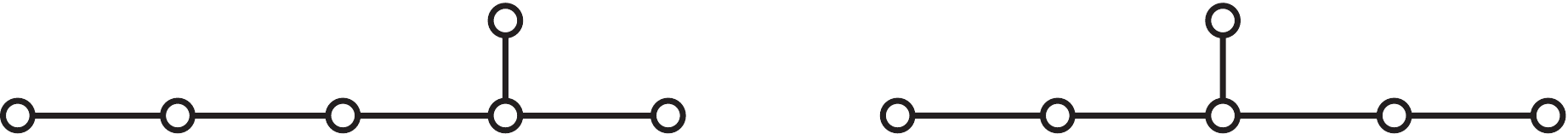}}\\ \vspace{-2pt}
$\null$\qquad\qquad\qquad $G_1$ \qquad\qquad\qquad  \qquad\qquad$G_2$\qquad\qquad\qquad$\null$
\caption{  $\match(G_1)< \match(G_2)$ and $S(G_1) >  S(G_2)$} \label{fig:mSno}\vspace{-15pt}
 \end{center}
 \end{figure}
\end{ex} 

\begin{ex}\label{diamSno}  Let $G_1$ and $G_2$ be the graphs shown in Figure \ref{fig:diamSno}.  Then $\diam(G_1)=4< 5=\diam(G_2)$, but $S(G_1)\approx2.37406 >2.35254\approx S(G_2)$.  Examples with the  reverse relation are easy to find, such as  $\diam(K_{1,n-1})=2< 3=\diam(P_4)$ and $S(K_{1,3})<  S(P_4)$. \vspace{-3pt}
\begin{figure}[!ht]
\begin{center}
\scalebox{.4}{\includegraphics{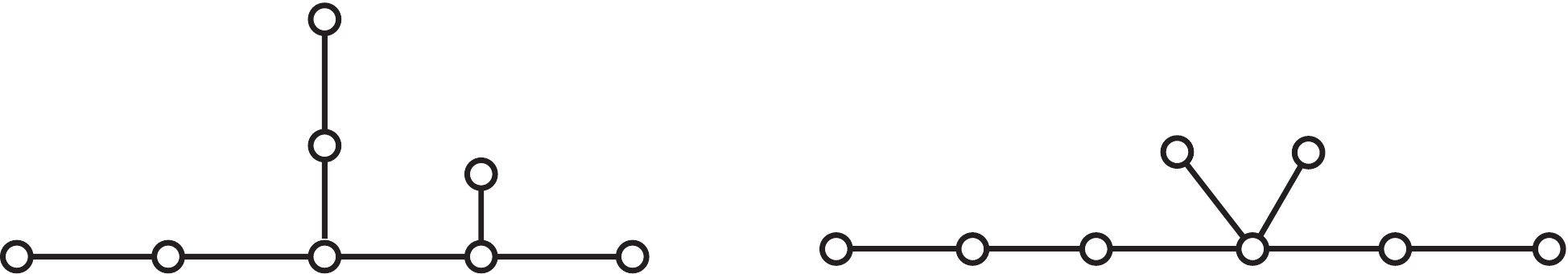}} \\\ \vspace{-2pt}
$\null$\qquad\qquad\qquad $G_1$ \qquad\qquad\qquad  \qquad\qquad$G_2$\qquad\qquad\qquad$\null$
\caption{  $\diam(G_1)< \diam(G_2)$ and $S(G_1) >  S(G_2)$}  \label{fig:diamSno}\vspace{-10pt}
 \end{center}
 \end{figure}
\end{ex} 

\begin{ex}\label{maxSno}  Let $G_1$ and $G_2$ be the graphs shown in Figure \ref{fig:maxSno}.  Then $\Delta(G_1)=4< 5=\Delta(G_2)$, but $S(G_1)\approx2.26678 <2.27741\approx S(G_2)$.  Examples with the  reverse relation are easy to find, such as  $\Delta(K_{1,n-1})=n-1>2=\Delta(P_n)$ and $S(K_{1,n-1})<  S(P_n)$, for $n\ge 4$. 

\begin{figure}[!ht]
\begin{center}
\scalebox{.4}{\includegraphics{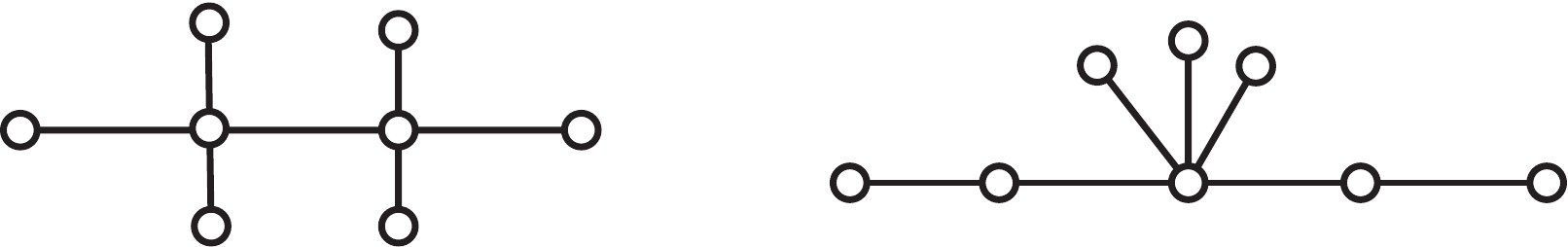}} \\\vspace{-2pt}
$\null$\qquad\qquad $G_1$ \qquad\qquad\qquad \qquad$G_2$\qquad\qquad$\null$
\caption{$\Delta(G_1)<\Delta(G_2)$ and $S(G_1) <  S(G_2)$} \label{fig:maxSno}\vspace{-15pt}
 \end{center}
 \end{figure}
\end{ex} 

Early in the development of spectral graph theory it was asked whether there exist nonisomorphic {\em cospectral} graphs, i.e., graphs having the same spectrum (for a particular matrix associated with the graph).  For each of the matrices associated with a graph, such as the adjacency and Laplacian matrices, nonisomorphic cospectral have been found.  Thus it is natural to ask whether there exist  noncospectral graphs having the same von Neumann entropy, i.e., {\em coentropy} graphs.  A search with {\em Sage}  produced numerous examples of order eight coentropy graphs having different spectra, including those in Example \ref{coent}.

\begin{ex}\label{coent}  Let $G$ a be the graph shown in Figure \ref{fig:coent}.  Then $S(G)=\log_214-\frac{4}{7}\log_2 8 =S(K_{2,6})$, but $\spec(\rho(G))=\{\frac 1 3,\frac 1 {6}^{(2)}, \frac 1 {8}^{(2)},\frac 1 {24}^{(2)},0\}$ whereas $\spec(\rho(K_{2,6}))=\left\{\frac 1 3,\frac 1 4,\frac 1 {12}^{(5)},0\right\}$.   \end{ex} 

\begin{figure}[!ht]
\begin{center}
\scalebox{.4}{\includegraphics{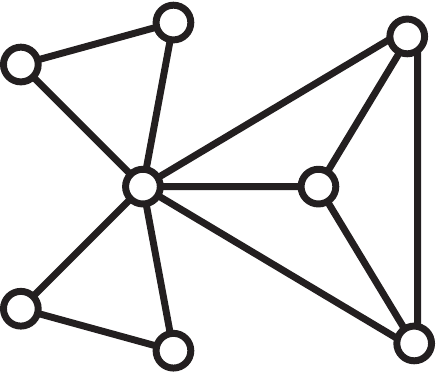}} \\\vspace{-2pt}
\caption{A graph $G$ that has the same von Neumann entropy as $K_{2,6}$ but a different spectrum} \label{fig:coent}\vspace{-10pt}
 \end{center}
 \end{figure}


 \section{Conclusion}

The behavior of von Neumann entropy is challenging to understand.  While many rules, such as `adding an edge raises entropy' work `most of the time,' as we saw in Proposition \ref{addedgelowerent} adding an edge can decrease von Neumann entropy. Thus the R\'enyi-Quantum Star Test, which works for almost all graphs, seems natural for entropy.  Understanding those graphs that fail this test may help to prove Conjecture \ref{Neumannmin}.

\begin{problem}\label{failRQST} Characterize graphs that fail the R\'enyi-Quantum Star Test.  \end{problem}
We make the following observations on graphs of order at most eight that fail the R\'enyi-Quantum Star Test:  \vspace{-3pt}
\ben
\item All those that fail have a leaf (degree one vertex).  \vspace{-3pt}
\item All those that fail are planar.\vspace{-3pt}
\een
\vspace{-3pt} Another approach to prove  Conjecture \ref{Neumannmin} would be to establish Conjecture \ref{Renyimin}.\medskip

As noted in Section \ref{scompare} we have not managed to find an interesting parameter that has nice correlation with (and is not trivially related to)  the von Neumann entropy, i.e. a parameter $\beta$ such that for any two graphs $G$ and $H$, $\beta(G)>\beta(H)$ implies $S(G)>S(H)$.

\begin{problem}\label{graphparam} Identify some interesting graph parameter(s) $\beta(G)$ such that $\beta(G)>\beta(H)$ implies $S(G)>S(H)$.
\end{problem}


\end{document}